\documentclass[12pt, reqno]{amsart}
\usepackage{}
\usepackage{stmaryrd}
\usepackage{amsmath}
\usepackage{amsfonts}
\usepackage{amssymb}
\usepackage[all,cmtip]{xy}           
\usepackage{xspace}
\usepackage{bbding}
\usepackage{txfonts}
\usepackage[shortlabels]{enumitem}
\usepackage{cite}
\usepackage{longtable}
\usepackage{makecell}
\usepackage{tikz}
\usepackage{titletoc}

\usepackage{ifpdf}
\ifpdf
 \usepackage[colorlinks,final,hyperindex]{hyperref}
\else
 \usepackage[colorlinks,final,backref=page,hyperindex,hypertex]{hyperref}
\fi
\usepackage[active]{srcltx}

\makeatletter

\begin{document}
\newcommand {\emptycomment}[1]{} 

\baselineskip=15pt
\newcommand{\nc}{\newcommand}
\newcommand{\delete}[1]{}
\nc{\mfootnote}[1]{\footnote{#1}} 
\nc{\todo}[1]{\tred{To do:} #1}

\nc{\mlabel}[1]{\label{#1}}  
\nc{\mcite}[1]{\cite{#1}}  
\nc{\mref}[1]{\ref{#1}}  
\nc{\meqref}[1]{\eqref{#1}} 
\nc{\mbibitem}[1]{\bibitem{#1}} 

\delete{
\nc{\mlabel}[1]{\label{#1}  
{\hfill \hspace{1cm}{\bf{{\ }\hfill(#1)}}}}
\nc{\mcite}[1]{\cite{#1}{{\bf{{\ }(#1)}}}}  
\nc{\mref}[1]{\ref{#1}{{\bf{{\ }(#1)}}}}  
\nc{\meqref}[1]{\eqref{#1}{{\bf{{\ }(#1)}}}} 
\nc{\mbibitem}[1]{\bibitem[\bf #1]{#1}} 
}

\newcommand {\comment}[1]{{\marginpar{*}\scriptsize\textbf{Comments:} #1}}
\nc{\mrm}[1]{{\rm #1}}
\nc{\id}{\mrm{id}}  \nc{\Id}{\mrm{Id}}
\nc{\admset}{\{\pm x\}\cup (-x+K^{\times}) \cup K^{\times} x^{-1}}

\def\a{\alpha}
\def\ad{associative D-}
\def\padm{$P$-allowable~}
\def\asi{ASI~}
\def\aybe{aYBe~}
\def\b{\beta}
\def\bd{\boxdot}
\def\bbf{\overline{f}}
\def\bF{\overline{F}}
\def\bbF{\overline{\overline{F}}}
\def\bbbf{\overline{\overline{f}}}
\def\bg{\overline{g}}
\def\bG{\overline{G}}
\def\bbG{\overline{\overline{G}}}
\def\bbg{\overline{\overline{g}}}
\def\bT{\overline{T}}
\def\bt{\overline{t}}
\def\bbT{\overline{\overline{T}}}
\def\bbt{\overline{\overline{t}}}
\def\bR{\overline{R}}
\def\br{\overline{r}}
\def\bbR{\overline{\overline{R}}}
\def\bbr{\overline{\overline{r}}}
\def\bu{\overline{u}}
\def\bU{\overline{U}}
\def\bbU{\overline{\overline{U}}}
\def\bbu{\overline{\overline{u}}}
\def\bw{\overline{w}}
\def\bW{\overline{W}}
\def\bbW{\overline{\overline{W}}}
\def\bbw{\overline{\overline{w}}}
\def\btl{\blacktriangleright}
\def\btr{\blacktriangleleft}
\def\calo{\mathcal{O}}
\def\ci{\circ}
\def\d{\delta}
\def\dd{\diamondsuit}
\def\D{\Delta}
\def\frakB{\mathfrak{B}}
\def\G{\Gamma}
\def\g{\gamma}
\def\gg{\mathfrak{g}}
\def\hh{\mathfrak{h}}
\def\k{\kappa}
\def\l{\lambda}
\def\ll{\mathfrak{L}}
\def\lh{\leftharpoonup}
\def\lr{\longrightarrow}
\def\N{Nijenhuis~}
\def\o{\otimes}
\def\om{\omega}
\def\opa{\cdot_{A}}
\def\opb{\cdot_{B}}
\def\p{\psi}
\def\r{\rho}
\def\ra{\rightarrow}
\def\rep{representation~}
\def\rh{\rightharpoonup}
\def\rr{{\Phi_r}}
\def\s{\sigma}
\def\st{\star}
\def\ss{\mathfrak{sl}_2}
\def\ti{\times}
\def\tl{\triangleright}
\def\tr{\triangleleft}
\def\v{\varepsilon}
\def\vp{\varphi}
\def\vth{\vartheta}
\def\wn{\widetilde{N}}
\def\wb{\widetilde{\beta}}

\newtheorem{thm}{Theorem}[section]
\newtheorem{lem}[thm]{Lemma}
\newtheorem{cor}[thm]{Corollary}
\newtheorem{pro}[thm]{Proposition}
\theoremstyle{definition}
\newtheorem{defi}[thm]{Definition}
\newtheorem{ex}[thm]{Example}
\newtheorem{rmk}[thm]{Remark}
\newtheorem{pdef}[thm]{Proposition-Definition}
\newtheorem{condition}[thm]{Condition}
\newtheorem{question}[thm]{Question}
\renewcommand{\labelenumi}{{\rm(\alph{enumi})}}
\renewcommand{\theenumi}{\alph{enumi}}

\nc{\ts}[1]{\textcolor{purple}{MTS:#1}}
\font\cyr=wncyr10


\title{Admissible Yang-Baxter equation for Nijenhuis perm algebras}

 \author[Ma]{Tianshui Ma\textsuperscript{*}}
 \address{School of Mathematics and Statistics, Henan Normal University, Xinxiang 453007, China}
         \email{matianshui@htu.edu.cn}
 \author[Song]{Feiyan Song}
 \address{School of Mathematics and Statistics, Henan Normal University, Xinxiang 453007, China}
         \email{songfeiyan@stu.htu.edu.cn}

 \thanks{\textsuperscript{*}Corresponding author}

\date{\today}

 \begin{abstract}
 In this paper, on one hand, based on the classical perm Yang-Baxter equation, we investigate under what conditions a perm algebra must be a Nijenhuis perm algebra. On the other hand, we derive the compatible conditions between classical perm Yang-Baxter equation and Nijenhuis operator by a class of Nijenhuis perm bialgebras.
 \end{abstract}

\subjclass[2020]{
17B38, 
16T10,   
16T25.   
}

\keywords{Nijenhuis operator; perm bialgebra; Yang-Baxter equation}

\maketitle




\allowdisplaybreaks

\section{Introduction}
 Nijenhuis operators are very useful in the deformation theory of algebras, bialgebras theory, integrable systems, almost complex structures, Nijenhuis geometry, etc.

 Perm algebras are a class of associative algebras, namely, a (right) perm algebra is a triple $(\gg,\cdot)$, where $\gg$ is a linear space. $\cdot: \gg\o \gg \longrightarrow \gg$ is a linear map, such that, for all $x,y,z\in \gg$
 \begin{eqnarray}\label{eq:ft}
 (xy)z=x(yz)=x(zy).
 \end{eqnarray}
 For example, let $\gg=K\{e_1,e_2\}$ be a 2-dimensional vector space satisfying the operation $e_1 e_1=e_1$ and $e_2 e_1=e_2$, then $\gg$ is a perm algebra.

 In 2002, Kolesnikov and Sartayev \cite{KS} obtained the necessary and sufficient conditions for a left-symmetric algebra to be embeddable into a differential perm-algebra. More recently, in 2004, Mashurov and Sartayev \cite{MS} proved that any metabelian Lie algebra can be embedded into an algebra in the subvariety of perm algebras. Also in 2024, Hou \cite{Hou} studied the extending structures by the unified product for perm algebras, and gave some properties for non-abelian extension as a special extending structure. He also introduced perm bialgebras, equivalently characterized by Manin triples of perm algebras and certain matched pairs of perm algebras, especially focused on coboundary perm bialgebras leading to the perm Yang-Baxter equation (there called "$\mathcal{S}$-equation"). In 2025, Lin, Zhou and Bai \cite{LZB} proved that there is a bialgebra structure for a perm algebra, i.e., perm bialgebra, or for a pre-Lie algebra, i.e., pre-Lie bialgebras, which could be characterized by the fact that its affinization by a quadratic pre-Lie algebra or a quadratic perm algebra respectively gives an infinite-dimensional Lie bialgebra. They also studied the perm Yang-Baxter equation and its solutions. We note that a perm algebra in \cite{Hou} is a right perm algebra, while a perm algebra in \cite{LZB} is a left prem algebra.

 According to the associative Yang-Baxter equation induced by coboundary ASI biaglebras, in 2024, the first named author and Long \cite{MLo} gave a method to obtain a Nijenhuis associative algebra and also considered the compatible conditions between the associative Yang-Baxter equation and the Nijenhuis operator via Nijenhuis ASI bialgebras. Though a perm algebra must be an associative algebra, the Yang-Baxter equation given in \cite{Hou,LZB} for perm algebras are {\bf essentially different} from the associative Yang-Baxter equation for associative algebras. In view of this fundamental difference, in this paper, we pay attention to the study for perm Yang-Baxter equation and Nijenhuis perm algebra.

 The paper is organized as follows. In Section \ref{se:npb}, we introduce the notions of a \N perm algebra and a representation of a \N perm algebra by infinitesimal deformation. Then, based on the results about perm bialgebras in \cite{Hou}, we present the notion of a \N perm bialgebra and give the equivalent characterizations by using a matched pair of Nijenhuis perm algebras and a Manin triple of \N Frobenius perm algebra. We focus on the coboundary case of \N perm bialgebra, then obtain $S$-admissible perm Yang-Baxter equations. We also investigate that how we can obtain \N perm bialgebras in terms of $\mathcal{O}$-operators on \N perm algebras. We end this section with the classification of 2-dimensional quasitriangular noncommutative perm bialgebras. In Section \ref{se:n}, we derive a method to obtain a \N operator on a perm (co)algebra highly depending on the classical perm Yang-Baxter equation introduced in \cite{Hou}. Some examples are also given.\\

 \noindent{\bf Notations:} Throughout this paper, we fix a field $K$. All vector spaces, tensor products, and linear homomorphisms are over $K$. We denote by $\id_M$ the identity map from $M$ to $M$, by $\tau$ the flip map.

 \section{$S$-admissible perm Yang-Baxter equations} \label{se:npb} In this section, we mainly find the compatible conditions between perm Yang-Baxter equation and \N operator by establishing \N perm bialgebras.

\subsection{Nijenhuis perm algebras and representations} We introduce the notions of a \N operator on a perm algebra and a representation of a \N perm algebra by means of infinitesimal deformation.

 \begin{defi} \cite{Hou,LZB} \label{de:l} Let $(\gg, \cdot)$ be a perm algebra, $V$ be a vector space and $\ell, r: \gg\lr End(V)$ be two linear maps. The triple $(V, \ell, r)$ is called a {\bf representation} of $\gg$, if, for all $x,y \in \gg$ and $v\in V$, the  equations below hold:
 \begin{eqnarray}
 &vr(xy)=(vr(x))r(y)=(vr(y))r(x),&\label{eq:1}\\
 &\ell(xy)v=\ell(x)(\ell(y)v)=\ell(x)(vr(y))=(\ell(x)v)r(y).&\label{eq:2j}
 \end{eqnarray}

 Let $(V_1,\ell_1,r_1)$ and $(V_2,\ell_2,r_2)$ be two representations of two perm algebras $(\gg,\cdot)$ and $(\mathfrak{h},\circ)$, respectively. A {\bf homomorphism} from $(V_1,\ell_1,r_1)$ to $(V_2,\ell_2,r_2)$ consists of a perm algebra homomorphism $\phi: \gg\rightarrow \mathfrak{h}$ and a linear map $f: V_1\rightarrow V_2$ such that the following conditions are satisfied:
 \begin{eqnarray}
 f(\ell_1(x)v)=\ell_2(\phi(x))f(v),~~f(v r_1(x))=f(v)r_2(\phi(x)),
 \end{eqnarray}
 where $x\in g$ and $v\in V_1$.
 \end{defi}

 \begin{lem}\label{lem:bgu} Let $\mu, m:\gg\o \gg\lr \gg$ (we write $\mu(x\o y)=x\cdot y$, $m(x\o y)=x\bullet y$), $\eta^{\ell}, \eta^{r}, \tau^{\ell}, \tau^{r}: \gg\lr End(V)$ be linear maps and $s, t$ be parameters. Define the following new multiplication and linear maps:
 \begin{eqnarray*}
 &x\cdot_{s,t} y=s x\cdot y+t x\bullet y, ~~{\ell}^{s,t}=s\eta^{\ell}+t\tau^{\ell}, ~~r^{s,t}=s\eta^{r}+t\tau^{r},&
 \end{eqnarray*}
 where $x, y\in \gg$. Then
 \begin{enumerate}[(1)]
   \item $(\gg,\cdot_{s,t})$ is a perm algebra if and only if $(\gg,\cdot)$, $(\gg, \bullet)$ are perm algebras, and for all $x,y,z\in \gg$,
   \begin{eqnarray*}
   (x\bullet y)\cdot z+(x\cdot y)\bullet z=x\cdot (y\bullet z)+x\bullet (y\cdot z)=x\cdot (z\bullet y)+x\bullet (z\cdot y).
   \end{eqnarray*}
   \item $(V,{\ell}^{s,t},r^{s,t})$ is a representation of $(\gg,\cdot_{s,t})$ if and only if $(V,\eta^{\ell},\eta^{r})$ is a representation of $(\gg,\cdot)$, $(V,\tau^{\ell},\tau^{r})$ is a representation of $(\gg,\bullet)$ and for all $x,y \in \gg$, $u\in V$,
   \begin{eqnarray*}
   &&\eta^{\ell}(x\bullet y)u+\tau^{\ell}(x\cdot y)u=\eta^{\ell}(x)(\tau^{\ell}(y) u)+\tau^{\ell}(x) (\eta^{\ell}(y)u)\\
   &&\hspace{38mm}=\eta^{\ell}(x) (u\tau^r(y))+\tau^{\ell}(x)(u\eta^{r}(y))=(\tau^{\ell}(x)u)\eta^{r}(y)
   +(\eta^{\ell}(x)u)\tau^r(y),\\
   &&u\eta^r(x\bullet y)+u\tau^r(x\cdot y)=(u\tau^r(y))\eta^{r}(x)+ (u\eta^{r}(y))\tau^r(x)= (u\tau^r(x))\eta^{r}(y)+(u\eta^{r}(x))\tau^r(y).
   \end{eqnarray*}
 \end{enumerate}
 \end{lem}

 \begin{proof}
 Straightforward.
 \end{proof}

  \begin{thm}\label{thm:rl} Let $(V,\eta^{\ell},\eta^{r})$ be a representation of $(\gg, \cdot)$, $(V, {\ell}^{s,t}, {r}^{s,t})$ be a representation of $(\gg, \cdot_{s,t})$, $N\in End(\gg)$ and $\alpha\in End(V)$ be two linear maps. Then $(s\id_\gg+t N, s\id_V+t \a)$ is a homomorphism from $(V, {\ell}^{s,t}, {r}^{s,t})$ of $(\gg, \cdot_{s,t})$ to $(V, \eta^{\ell}, \eta^{r})$ of $(\gg, \cdot)$ if and only if for all $x, y\in \gg$,
 \begin{eqnarray}
 &x\bullet y=N(x)\cdot y+x\cdot N(y)-N(x\cdot y),&\label{eq:vb}\\
 &N(x\bullet y)=N(x)\cdot N(y),&\label{eq:yq}\\
 &\tau^{\ell}(x)=\eta^{\ell}(N(x))+\eta^{\ell}(x)\circ\alpha-\alpha\circ \eta^{\ell}(x),&\label{eq:lc}\\
 &\eta^{\ell}(N(x))\circ\alpha=\alpha \circ\tau^{\ell}(x),&\label{eq:xm}\\
 &\tau^r(x)=\eta^r(N(x))+\eta^r(x)\circ \alpha-\alpha\circ\eta^r(x),&\label{eq:fk}\\
 &\eta^r(N(x))\circ\alpha=\alpha\circ \tau^r(x).\label{eq:ty}
 \end{eqnarray}
 \end{thm}

 \begin{proof} It is direct.
 \end{proof}

 By Eqs.(\ref{eq:vb}) and (\ref{eq:yq}), we have:
 \begin{defi}\label{de:5t} A {\bf Nijenhuis perm algebra} is a pair $(\gg, N)$ including a perm algebra $\gg$ and a Nijenhuis operator $N$ on $\gg$, i.e, a linear map $N:\gg\lr \gg$ satisfies
 \begin{eqnarray}\label{eq:ggg}
 N(x)N(y)+{N^2}(xy)=N(N(x)y)+N(xN(y)),
 \end{eqnarray}
 where $x, y\in \gg$.
 \end{defi}

 By Eqs.(\ref{eq:lc})-(\ref{eq:ty}), one can obtain:
 \begin{defi}\label{de:a} A {\bf representation} of a Nijenhuis perm algebra $(\gg, N)$ is a quadruple $(V, \ell, r, \alpha)$, where $(V,\ell,r)$ is a representation of $(\gg, \cdot)$ and $\alpha:V\lr V$ is a linear map, such that, for all $x \in \gg, v\in V$,
 \begin{eqnarray}
 &\ell((N(x))\alpha(v)+\alpha^2(\ell(x)v)=\alpha(\ell(N(x))v)+\alpha(\ell(x)\alpha (v)),&\label{eq:1p}\\
 &\alpha(v)r(N(x))+\alpha^2(vr(x))=\alpha(v r(N(x)))+\alpha(\alpha(v)r(x)).&\label{eq:1i}
 \end{eqnarray}
 A linear map $f:V_1\lr V_2$ is called a {\bf homomorphism} from the representations $(V_1,\ell_1,r_1,\alpha_1)$ to $(V_2,\ell_2,r_2,\alpha_2)$ of $(\gg,\cdot)$ if for all $x\in \gg, u\in V_1$,
 \begin{eqnarray*}
 f(\ell_1(x)u)=\ell_2(x)f(u),~f(u r_1(x))=f(u)r_2(x),~f(\alpha_1(u))=\alpha_2(f(u)).
 \end{eqnarray*}
 If, further, $f$ is bijective, then we call these two representations $(V_1,\ell_1,r_1,\alpha_1)$ and $(V_2,\ell_2,r_2,\alpha_2)$ {\bf equivalent}.
 \end{defi}

 \begin{ex}\label{ex:gh} Let $(\gg,N)$ be a Nijenhuis perm algebra, then $(\gg,L,R,N)$ is a representation of $(\gg,N)$, which is called the {\bf adjoint representation} of $(\gg, N)$, where $L, R$ denote the left and right multiplications respectively.
 \end{ex}

 \begin{pro}\label{pro:B} Let $(\gg, N)$ be a Nijenhuis perm algebra and $(V,\ell,r)$ be a representation of $\gg$, $\alpha:V\lr V$ be a linear map. Define a linear map $N+\alpha:\gg\oplus V \lr \gg\oplus V$ by
 \begin{eqnarray}\label{eq:d}
 (N+\alpha)(x+u)=N(x)+\alpha(u),
 \end{eqnarray}
 and a linear operation by
 \begin{eqnarray}\label{eq:1w}
 (x+u)\ast(y+v)=x y+\ell(x)v+u r(y),
 \end{eqnarray}
 then $(\gg\oplus V, N+\alpha)$ is a Nijenhuis perm algebra if and only if $(V,\ell,r,\alpha)$ is a representation of $(\gg,N)$. This Nijenhuis perm algebra is called the {\bf semi-direct product} of $(\gg,N)$ associated with $(V,\ell,r,\alpha)$, denoted by $(\gg\ltimes_{\ell,r}V, N+\alpha)$.
 \end{pro}

 \begin{proof} First by \cite[Proposition 2.5]{Hou} and $(V,\ell,r)$ is a representation of $\gg$, $\gg\oplus V$ is a perm algebra. Next, by a similar discussion in \cite[Proposition 3.6]{MLo}, we can obtain that $(\gg\oplus V,N+\alpha)$ is a Nijenhuis perm algebra if and only if $(V,\ell,r,\alpha)$ is a representation of $(\gg,N)$.
 \end{proof}

 Next we introduce the dual representation of a representation of a \N perm algebra.
 \begin{lem}\label{lem:HO} Let $(\gg,N)$ be a Nijenhuis perm algebra, $(V,\ell,r)$ be a representation of $\gg$ and $\beta: V\lr V$ be a linear map. Then $(V^*,r^*-\ell^*,r^*,\beta^*)$ is a representation of $(\gg,N)$, where ${\ell}^*,r^*:\gg\lr End(V^*)$ by
 \begin{eqnarray*}
 \langle {\ell^*}(x)u^*,v\rangle=\langle u^*,{\ell(x)}v\rangle,
 \langle u^*r^*(x),v\rangle=\langle u^*,vr(x)\rangle,~\forall~x\in \gg, u^*\in V^*, v\in V \end{eqnarray*}
 if and only if for all $x\in \gg, u\in V$, the following equations hold:
 \begin{eqnarray}
 &\beta(\ell(N(x))u)+\ell(x)\beta^2(u)-\ell((N(x))\beta(u)-\beta(\ell(x) \beta(u))=0,&\label{eq:fb}\\
 &\beta(u r(N(x)))+\beta^2(u)r(x)-\beta(u)r((N(x))-\beta(\beta(u)r(x))=0.&\label{eq:fs}
 \end{eqnarray}
 \end{lem}

 \begin{proof} We note that $(V^*,r^*-\ell^*,r^*)$ is a representation of $\gg$ by \cite[Lemma 4.2]{Hou}. For all $x\in \gg, u^*\in V^*, v\in V$, we have
 \begin{eqnarray*}
 &&\langle\beta^*(u^*)r^*(N(x))+{\beta^*}^2(u^*r^*(x))-\beta^*(u^*r^*(N(x)))-\beta^*(\beta^*(u^*)r^*(x)), v\rangle\\
 &&=\langle u^*, \beta(vr(N(x)))+\beta^2(v)r(x)-\beta(v)r(N(x))-\beta(\beta(v)r(x))\rangle.
 \end{eqnarray*}
 Similarly,
 \begin{eqnarray*}
 &&\langle (r^*-\ell^*)(N(x))\beta^*(u^*)+{\beta^*}^2((r^*-\ell^*)(x)u^*)-\beta^*((r^*-\ell^*)(N(x))u^*)-\beta^*((r^*-\ell^*)(x) \beta^*(u^*)), v\rangle\\
 &&=\langle u^*, \beta(vr(N(x)))+\beta^2(v)r(x)-\beta(v)r(N(x))-\beta(\beta(v)r(x))\\
 &&\quad-\beta(\ell(N(x))v)-\ell(x)\beta^2(v)+\ell(N(x))\beta(v)+\beta(\ell(x)\beta(v))\rangle.
 \end{eqnarray*}
 Thus Eqs.(\ref{eq:1p}) and (\ref{eq:1i}) hold for $\beta^*$ if and only if Eqs.(\ref{eq:fb}) and (\ref{eq:fs}) hold.
 \end{proof}

 \begin{defi}\label{de:mrt} Let $(\gg,N)$ be a Nijenhuis perm algebra, $(V, \ell,r)$ be a representation of $\gg$, and $\beta:V\lr V$ be a linear map. If Eqs.(\ref{eq:fb}) and (\ref{eq:fs}) hold, then we call that {\bf $\beta$ is admissible to Nijenhuis perm algebra $(\gg,N)$ on $(V, \ell,r)$} or {\bf $(\gg,N)$ is $\beta$-admissible on $(V, \ell,r)$}. When $(V, \ell, r)=(\gg, L, R)$, we say that {\bf $\beta$ is admissible to  $(\gg,N)$} or {\bf $(\gg,N)$ is $\beta$-admissible}.
 \end{defi}

 Let $(V, \ell,r)=(\gg,L,R)$, we have
 \begin{cor}\label{cor:mt} Let $(\gg,N)$ be a Nijenhuis perm algebra. A linear map $S:\gg\lr \gg$ is admissible to $(\gg,N)$ if and only if, for all $x,y\in \gg$
 \begin{eqnarray}
 &S(N(x)y)+xS^2(y)-N(x)S(y)-S(xS(y))=0,&\label{eq:sfh}\\
 &S(xN(y))+S^2(x)y-S(x)N(y)-S(S(x)y)=0.&\label{eq:sh}
 \end{eqnarray}
 \end{cor}

 \subsection{\N perm bialgebras and equivalent characterizations}\mlabel{se:tri}  Dual to the notion of perm algebras, we have: A {\bf perm coalgebra} is a pair $(\gg, \D)$, where $\gg$ is a vector space and $\D: \gg\lr \gg\o \gg$ (Sweedler notation \cite{Sw}: $\D(x)=x_{(1)}\o x_{(2)}$) is a linear map such that for all $x\in\gg$,
 \begin{eqnarray}\label{eq:yds}
 x_{(1)(1)}\o x_{(1)(2)}\o x_{(2)}=x_{(1)}\o x_{(2)(1)}\o x_{(2)(2)}=x_{(1)}\o x_{(2)(2)}\o x_{(2)(1)}.
 \end{eqnarray}

 \begin{pro}\mlabel{pro:de:cl}
  \begin{enumerate}[(1)]
   \item \mlabel{it:de:cl1} If $(\gg, \D)$ is a perm coalgebra, then $(\gg^*, \D^*)$ is a perm algebra, where $\D^*:\gg^*\o \gg^*\lr \gg^*$ is defined by
   $$
   \langle\D^*(\xi\o \eta), x\rangle=\langle \xi, x_{(1)}\rangle \langle \eta, x_{(2)}\rangle, \quad \forall~~\xi, \eta\in \gg^*, x\in \gg.
   $$
   \item \mlabel{it:de:cl2} If $(\gg, \cdot)$ is a finite-dimensional perm algebra, then $(\gg^*, \cdot^*)$ is a perm coalgebra, where $\cdot^*:\gg^*\lr \gg^*\o \gg^*$ is defined by
   $$
   \langle\cdot^*(\xi), x\o y\rangle=\langle \xi, x\cdot y\rangle, \quad \forall~~\xi\in \gg^*, x, y\in \gg.
   $$
 \end{enumerate}
 \end{pro}

 \begin{proof} Similar to the case of Hopf algebra, see \cite{Sw}. \end{proof}

 Let recall from \cite{Hou,LZB} the definition of a perm bialgebra. Let $\gg$ be a perm algebra. A {\bf perm bialgebra} is a triple $(\gg, \cdot, \D)$ such that $(\gg,\D)$ is a perm coalgebra and, for all $x, y\in\gg$, the compatibility conditions below hold:
 \begin{eqnarray}\label{eq:hh}
  \D(xy)-\tau\D(xy)
  \hspace{-3mm}&=&\hspace{-3mm}\D(yx)-\tau\D(yx)\nonumber \\
  \hspace{-3mm}&=&\hspace{-3mm}(R(y)\o \id)\D(x)-\tau(\id \o R(y))\D(x)+(\id\o R(x))\D(y)-\tau(R(x)\o \id)\D(y),\qquad\\
 \D(xy)\hspace{-3mm}&=&\hspace{-3mm}(L(x)\o \id)\D(y)+(\id \o R(y))\D(x)-(\id\o L(y))\D(x)\nonumber\\
 \hspace{-3mm}&=&\hspace{-3mm}(L(x)\o \id)\D(y)+\tau(\id \o L(x))\D(y)+(\id\o R(y))\D(x)\\
 \hspace{-3mm}&=&\hspace{-3mm}(R(y)\o \id)\D(x)+(\id \o R(x))\D(y)-(\id\o L(x))\D(y)\nonumber\\
 &&\hspace{5mm}-\tau(R(x)\o \id)\D(y)+\tau(L(x)\o \id )\D(y).\nonumber\label{eq:kk}
 \end{eqnarray}

 \subsubsection{Matched pairs of Nijenhuis perm algebras}
 \begin{pro} \cite{Hou} \label{pro:gio} Let $\gg$, $\mathfrak{h}$ be two perm algebras, $(\mathfrak{h},\ell_\gg,r_\gg)$ be a representation of $\gg$ and $(\gg,\ell_\mathfrak{h},r_\mathfrak{h})$ be a representation of $\mathfrak{h}$. Define a bilinear operation $\star$ on $\gg\oplus \mathfrak{h}$ given by
 \begin{eqnarray}\label{eq:cd}
 (a+x)\star(b+y)=ab+\ell_\mathfrak{h}(x)b+ar_\mathfrak{h}(y)+xy+\ell_\gg(a)y+xr_\gg(b), ~\forall~a,b\in \gg,x,y\in \mathfrak{h}.
 \end{eqnarray}
 Then $\gg\oplus \mathfrak{h}$ is a perm algebra if and only if $(\gg,\mathfrak{h},\ell_\gg, r_\gg,\ell_\mathfrak{h},r_\mathfrak{h})$ is a matched pair of perm algebras, that is to say, for all $x,x_1,x_2 \in \gg$, $y,y_1,y_2 \in \mathfrak{h}$, the following equations hold:
 \begin{eqnarray*}
 &&\ell_\mathfrak{h}(y)(x_1x_2)=\ell_\mathfrak{h}(y)(x_2x_1)=(\ell_\mathfrak{h}(y)(x_1))x_2+\ell_\mathfrak{h}(y r_\gg(x_1))(x_2),\label{eq:y7o}\\
 &&(x_1x_2)r_\mathfrak{h}(y)=x_1(x_2 r_\mathfrak{h}(y))+x_1 r_\mathfrak{h}(\ell_\gg(x_2)(y))\\\nonumber
 &&\hspace{19mm}=x_1(\ell_\mathfrak{h}(y)(x_2))+x_1 r_\mathfrak{h}(y r_\gg(x_2))=(x_1 r_\mathfrak{h}(y))x_2+\ell_\mathfrak{h}(\ell_\gg(x_1)(y))(x_2),\label{eq:yo}\\
 &&\ell_\gg(x)(y_1y_2)=\ell_\gg(x)(y_2y_1)=(\ell_\gg(x)(y_1))y_2+\ell_\gg(x r_\mathfrak{h}(y_1))(y_2),\label{eq:yp}\\
 &&(y_1y_2)r_\gg(x)=y_1(y_2 r_\gg(x))+y_1 r_\gg(\ell_\mathfrak{h}(y_2)(x))\\\nonumber
 &&\hspace{19mm}=y_1(\ell_\gg(x)(y_2))+y_1 r_\gg(x r_\mathfrak{h}(y_2))=y_2(y_1 r_\gg(x))+\ell_\gg(\ell_\mathfrak{h}(y_1)(x))(y_2).\label{eq:yo}
 \end{eqnarray*}
 \end{pro}

 We extend the notion of a matched pair of perm algebras to \N case.
 \begin{defi}\label{de:mv}  A {\bf matched pair of Nijenhuis perm algebras $(\gg,N_\gg)$ and $(\mathfrak{h},N_\mathfrak{h})$} is a hexatuple $((\gg,N_\gg), (\mathfrak{h}, N_\mathfrak{h}),\ell_\gg, r_\gg,\ell_\mathfrak{h}, r_\mathfrak{h})$, where $(\mathfrak{h}, \ell_\gg, r_\gg, N_\mathfrak{h})$ is a representation of $(\gg, N_\gg)$ and $(\gg, \ell_\mathfrak{h}, r_\mathfrak{h}, N_\gg)$ is a representation of $(\mathfrak{h}, N_\mathfrak{h})$, and $(\mathfrak{h}, \gg, \ell_\gg, r_\gg, \ell_\mathfrak{h}, r_\mathfrak{h})$ is a matched pair of $\gg$ and $\mathfrak{h}$.
 \end{defi}

 \begin{thm}\label{thm:gp} Let $(\gg,N_\gg)$ and $(\mathfrak{h},N_\mathfrak{h})$ be two Nijenhuis perm algebras, $(\gg, \mathfrak{h}, \ell_\gg, r_\gg, \ell_\mathfrak{h}, r_\mathfrak{h})$ be a matched pair of $\gg$ and $\mathfrak{h}$. Define a linear map $N_{\gg+\mathfrak{h}}:\gg \oplus \mathfrak{h} \lr \gg \oplus \mathfrak{h}$ by
 \begin{eqnarray}\label{eq:cgy}
 (N_{\gg+\mathfrak{h}})(a+x)=N_\gg(a)+N_\mathfrak{h}(x), ~\forall a\in \gg,x\in \mathfrak{h},
 \end{eqnarray}
 and a linear operation $\star :(\gg \oplus \mathfrak{h})\times (\gg \oplus \mathfrak{h})\lr \gg \oplus \mathfrak{h}$ by Eq.(\ref{eq:cd}), then $(\gg \oplus \mathfrak{h},N_{\gg+\mathfrak{h}})$ is a Nijenhuis perm algebra if and only if $((\gg,N_\gg), (\mathfrak{h}, N_\mathfrak{h}), \ell_\gg,r_\gg, \ell_\mathfrak{h}, r_\mathfrak{h})$ is a matched pair of $(\gg,N_\gg)$ and $(\mathfrak{h},N_\mathfrak{h})$.
 \end{thm}

 \begin{proof} It can be proved by Proposition \ref{pro:gio} and \cite[Theorem 2.3.2]{MLo}.
 \end{proof}

 \subsubsection{Manin triple of a Nijenhuis perm algebra} Let us recall from \cite{Hou} that a bilinear form $\mathfrak{B}$ on a perm algebra $\gg$ is called {\bf invariant} if for all $x,y,z \in \gg$,
 \begin{eqnarray}\label{eq:hi}
 \mathfrak{B}(x y, z)=\mathfrak{B}(y, z x)-\mathfrak{B}(y, x z).
 \end{eqnarray}
 A {\bf Frobenius perm algebra} $(\gg,\mathfrak{B})$ is a perm algebra $\gg$ with a nondegenerate invariant bilinear form $\mathfrak{B}$. A Frobenius perm algebra $(\gg,\mathfrak{B})$ is called {\bf skew-symmetric} if $\mathfrak{B}$ is skew-symmetric.

 \begin{rmk}\label{rmk:kj} If $\mathfrak{B}$ is skew-symmetric, then for all $x,y,z \in \gg$, we have
 \begin{eqnarray}\label{eq:gm}
 \mathfrak{B}(xy,z)=\mathfrak{B}(x,zy).
 \end{eqnarray}
 In fact, by Eq.(\ref{eq:hi}), we have
 \begin{eqnarray*}
 \mathfrak{B}(x,zy)=-\mathfrak{B}(zy,x)=-(\mathfrak{B}(y,xz)-\mathfrak{B}(y,zx))=\mathfrak{B}(xy,z).
 \end{eqnarray*}
 \end{rmk}

 \begin{defi} \cite{Hou,LZB} \label{de:cf} Let $(\gg,\cdot)$ be a perm algebra. Suppose that there is a perm algebra structure $\circ$ on its dual space $\gg^*$, and a perm algebra structure on the direct sum $\gg \oplus \gg^*$ of the underlying vector spaces of $\gg$ and $\gg^*$ which contains both $(\gg,\cdot)$ and $(\gg^*,\circ)$ as subalgebras. Define a bilinear form on $\gg \oplus \gg^*$ by for all $x,y \in \gg$ and $a,b \in \gg^*$,
 \begin{eqnarray}\label{eq:lpo}
 \mathfrak{B}(x+a,y+b)=\langle x, b\rangle-\langle y, a\rangle
 \end{eqnarray}
 If $\mathfrak{B}$ is invariant, so that $(\gg \oplus \gg^*,\mathfrak{B})$ is a skew-symmetric Frobenius perm algebra, then the Frobenius perm algebra
 is called a {\bf (standard) Manin triple of Frobenius perm algebra} associated to  $(\gg,\cdot)$ and $(\gg^*,\circ)$, which is denoted by $((\gg \oplus \gg^*,\mathfrak{B}), \gg, \gg^*)$.
 \end{defi}

 \begin{lem} \cite{Hou} \label{lem:io} Let $(\gg,\cdot)$ be a  perm algebra. Suppose that there is a  perm algebra structure $(\gg^*,\circ)$ on its dual space $\gg^*$. Then there is a Manin triple of Frobenius perm algebra associated to $(\gg,\cdot)$ and  $(\gg^*,\circ)$ if and only if $(\gg,\gg^*,{R_\cdot}^*-{L_\cdot}^*,{R_\cdot}^*,{R_\circ}^*-{L_\circ}^*,{R_\circ}^*)$ is a matched pair of  perm algebras.
 \end{lem}

 We now extend these notions and properties to Nijenhuis perm algebras.

 \begin{defi}\label{de:cl} A {\bf Nijenhuis Frobenius perm algebra} (abbr. NF perm algebra) is a triple $(\gg, N, \mathfrak{B})$, where $(\gg, N)$ is a Nijenhuis perm algebra and $(\gg, \mathfrak{B})$ is a Frobenius perm algebra.
 \end{defi}

 Let $\widehat{N}:\gg\rightarrow \gg$ denote the adjoint linear transformation of $N$ under the nondegenerate bilinear form $\mathfrak{B}$:
 \begin{eqnarray}\label{eq:fg}
 \mathfrak{B}(N(x),y)=\mathfrak{B}(x,\widehat{N}(y)),~ \forall~ x, y \in \gg.
 \end{eqnarray}

 \begin{pro}\label{pro:kl} Let $(\gg,N,\mathfrak{B})$ be a skew-symmetric NF perm algebra. Then for the adjoint operator $\widehat{N}$ in Eq.(\ref{eq:fg}), the quadruple $(\gg^*,R^*-L^*,R^*,{\widehat{N}}^*)$ is a representation of  $(\gg,N)$. Moreover, $(\gg^*,R^*-L^*,R^*,{\widehat{N}}^*)$ is equivalent to $(\gg,L,R,N)$ as representations of $(\gg,N)$. Conversely, let $(\gg,N)$ be a Nijenhuis perm algebra and $S: \gg\lr \gg$ be a linear map that is admissible to $(\gg,N)$. If the representation $(\gg^*,R^*-L^*,R^*,S^*)$ of $(\gg,N)$ is equivalent to $(\gg,L,R,N)$, then there exists a nondegenerate  bilinear form $\mathfrak{B}$ such that $(\gg,N,\mathfrak{B})$ is an NF perm algebra for which $\widehat{N}=S$.
 \end{pro}

 \begin{proof} For all $x,y,z \in \gg$, we have
 \begin{eqnarray*}
 0&=&\mathfrak{B}(N(x)N(y)+{N^2}(xy)-N(N(x)y)-N(xN(y)),z)\\
 &=&\mathfrak{B}(N(x)N(y),z)+{\mathfrak{B}}({N^2}(xy),z)-\mathfrak{B}(N(N(x)y),z)-\mathfrak{B}(N(xN(y)),z)\\
 &\stackrel {(\ref{eq:gm})(\ref{eq:fg})}{=}&\mathfrak{B}(x,\widehat{N}(zN(y)))+\mathfrak{B}(x,{\widehat{N}}^2 (z)y)-\mathfrak{B}(x,\widehat{N}(\widehat{N}(z)y))-\mathfrak{B}(x,\widehat{N}(z)N(y))\\
 &=&\mathfrak{B}(x,\widehat{N}(zN(y))+{\widehat{N}}^2 (z)y-\widehat{N}(\widehat{N}(z)y)-\widehat{N}(z)N(y)).\\
 \end{eqnarray*}
 Then, we obtain
 \begin{eqnarray*}
 \widehat{N}(zN(y))+{\widehat{N}}^2 (z)y-\widehat{N}(\widehat{N}(z)y)-\widehat{N}(z)N(y)=0,
 \end{eqnarray*}
 i.e., Eq.$(\ref{eq:sh})$ holds. On the other hand, by Eqs.(\ref{eq:hi}) and (\ref{eq:fg}), we have:
 \begin{eqnarray*}\label{eq:iuo}
 \widehat{N}(zN(y))-\widehat{N}(N(y)z)+{\widehat{N}}^2 (z)y-y{\widehat{N}}^2(z)-\widehat{N}(z)N(y)+N(y)\widehat{N}(z)-\widehat{N}(\widehat {N}(z)y)+\widehat{N}(y\widehat{N}(z))=0.
 \end{eqnarray*} 
 Then Eq.(\ref{eq:sfh}) holds by Eq.(\ref{eq:sh}). Hence $(\gg^*,R^*-L^*,R^*,{\widehat{N}}^*)$ is a representation of $(\gg,N)$.

 Next, for all $x,y \in \gg $, we define the linear map $\psi:\gg \lr \gg^*$ by
 \begin{eqnarray*}\label{eq:ik}
 \psi(x)y:=\langle\psi(x),y\rangle=\mathfrak{B}(x,y).
 \end{eqnarray*}
 The nondegeneracy of $\mathfrak{B}$ gives the bijectivity of $\psi$. 
 Then by the definition of $\psi$, it is direct to prove that $(\gg^*,R^*-L^*,R^*,{\widehat{N}}^*)$ is equivalent to $(\gg,L,R,N)$ as representation of $(\gg,N)$.

 Conversely, suppose that $\psi:\gg\lr \gg^*$ is the linear isomorphism giving the equivalence between $(\gg,L,R,N)$ and $(\gg^*,R^*-L^*,R^*,S^*)$. Define a bilinear form $\mathfrak{B}$ on $\gg$ by $\mathfrak{B}(x,y)=\langle\psi(x),y\rangle$. Then a similar argument gives the NF perm algebra $(\gg,N,\mathfrak{B})$ and $\widehat{N}=S$.
 \end{proof}

 We now extend the notion of Manin triple to NF perm algebras.
 \begin{defi}\label{de:fu} Let $(\gg,\cdot,N)$ and $(\gg^*,\circ,S^*)$ be two Nijenhuis perm algebras. A {\bf Manin triple of an NF perm algebra} associated to $(\gg,\cdot,N)$ and $(\gg^*,\circ,S^*)$  is a Manin triple $((\gg \oplus \gg^*,\mathfrak{B}), \gg, \gg^*)$ of Frobenius perm algebra associated to $(\gg,\cdot)$ and $(\gg^*,\circ)$ such that $(\gg\oplus \gg^*, N+S^*, \mathfrak{B})$ is an NF perm algebra. We denote it by $((\gg\oplus \gg^*, N+S^*, \mathfrak{B}), (\gg,N), (\gg^*,S^*))$.
 \end{defi}

 \begin{lem}\label{lem:wb} Let $((\gg\oplus \gg^*, N+S^*, \mathfrak{B}), (\gg,N), (\gg^*,S^*))$ be a Manin triple of NF perm algebra associated to $(\gg, N)$ and $(\gg^*, S^*)$.
 \begin{enumerate}[(1)]
 \item \label{it:11} The adjoint $\widehat{N+S^*}$  of $N+S^*$ with respect to $\mathfrak{B}$ is $S+N^*$. Further $S+N^*$ is admissible to $(\gg\oplus \gg^*,N+S^*)$.
 \item \label{it:13} $S$ is admissible to $(\gg,N)$.
 \item \label{it:15} $N^*$ is admissible to $(\gg^*,S^*)$.
 \end{enumerate}
 \end{lem}

 \begin{proof} It is similar to \cite[Lemma 4.6]{MLo}. \end{proof}

 \begin{thm}\label{thm:hu} Let $(\gg,\cdot,N)$ be a Nijenhuis perm algebra. Suppose that there is a Nijenhuis perm algebra structure $(\gg^*,\circ,S^*)$ on its dual space $\gg^*$. Then there is a Manin triple $((\gg\oplus \gg^*, N+S^*, \mathfrak{B}), (\gg,N), (\gg^*,S^*))$ of NF perm algebra associated to $(\gg,\cdot,N)$ and $(\gg^*,\circ,S^*)$ if and only if $((\gg,N),(\gg^*,S^*),{R_\cdot}^*-{L_\cdot}^*,{R_\cdot}^*,{R_\circ}^*-{L_\circ}^*,{R_\circ}^*)$  is a matched pair of Nijenhuis perm algebras.
 \end{thm}

 \begin{proof} $(\Longrightarrow)$  The given Manin triple $((\gg\oplus \gg^*, N+S^*, \mathfrak{B}), (\gg,N), (\gg^*,S^*))$ associated to $(\gg,N)$ and $(\gg^*,S^*)$ implies that $((\gg\oplus \gg^*, \mathfrak{B}), \gg, \gg^*)$ is a Manin triple of Frobenius perm algebra associated to $\gg$ and $\gg^*$. Hence by Lemma \ref{lem:io}, $(\gg,\gg^*,{R_\cdot}^*-{L_\cdot}^*,{R_\cdot}^*,{R_\circ}^*-{L_\circ}^*,{R_\circ}^*)$ is a matched pair of perm algebras. Further, by Lemma \ref{lem:wb}, $(\gg^*,{R_\cdot}^*-{L_\cdot}^*,{R_\cdot}^*,S^*)$ is a representation of $(\gg, N)$ and $(\gg,{R_\circ}^*-{L_\circ}^*,{R_\circ}^*,N)$ is a representation of $(\gg^*, S^*)$. Hence, $(\gg,\gg^*,N,S^*,{R_\cdot}^*-{L_\cdot}^*,{R_\cdot}^*,{R_\circ}^*-{L_\circ}^*,{R_\circ}^*)$  is a matched pair of Nijenhuis perm algebras.

 $(\Longleftarrow)$ Let $(\gg,\gg^*,N,S^*,{R_\cdot}^*-{L_\cdot}^*,{R_\cdot}^*,{R_\circ}^*-{L_\circ}^*,{R_\circ}^*)$  be a matched pair of Nijenhuis perm algebras. Then $(\gg,\gg^*,{R_\cdot}^*-{L_\cdot}^*,{R_\cdot}^*,{R_\circ}^*-{L_\circ}^*,{R_\circ}^*)$  is a matched pair of perm algebras. By Lemma \ref{lem:io}, $(\gg \oplus \gg^*,\mathfrak{B})$ is a Frobenius perm algebra. By Theorem \ref{thm:gp}, the matched pair of Nijenhuis perm algebras also equips the perm algebra $\gg\oplus \gg^*$ with the Nijenhuis operator $N+S^*$, giving us an NF perm algebra.
 \end{proof}

 \subsubsection{Nijenhuis perm bialgebras}
 \begin{defi}
  A {\bf \N perm coalgebra} is a triple $(\gg, \D, S)$, where $(\gg, \D)$ is a perm coalgebra and $S$ is a \N operator on $(\gg, \D)$, i.e., the equality below holds:
 \begin{eqnarray}\mlabel{eq:gd}
 S(x_{(1)})\o S(x_{(2)})+S^2(x)_{(1)}\o S^2(x)_{(2)}=S(S(x)_{(1)})\o S(x)_{(2)}+S(x)_{(1)}\o S(S(x)_{(2)}),~~\forall~x\in \gg.
 \end{eqnarray}
 \end{defi}

 \begin{defi}\label{de:yy} A Nijenhuis perm bialgebra is a vector space $\gg$ together with linear maps $\cdot: \gg\o \gg\lr \gg$, $\D : \gg\lr \gg\o \gg$, $N, S: \gg\lr \gg$ such that
 \begin{enumerate}[(1)]
 \item $(\gg,\cdot,\D)$ is a perm bialgebra.
 \item $(\gg,\cdot,N)$ is a Nijenhuis perm algebra.
 \item $(\gg,\D,S)$ is a Nijenhuis perm coalgebra.
 \item $S$ is admissible to $(\gg,\cdot,N)$, that is, Eq.(\ref{eq:sfh}) and Eq.(\ref{eq:sh}) hold.
 \item $N^*$ is admissible to $(\gg^*,\D^*,S^*)$, that is, the equations below hold:
 \begin{eqnarray}
 &(S \o \id)\D N+(\id \o N^2)\D-(S\o N)\D -(\id \o N)\D N=0,&\label{eq:jy}\\
 &(\id \o S)\D N+(N^2 \o \id)\D-(N\o S)\D -(N \o \id)\D N=0.&\label{eq:jj}
 \end{eqnarray}
 \end{enumerate}
 We denote the Nijenhuis perm bialgebra by $((\gg,\cdot,N),\D,S)$.
 \end{defi}

 \begin{ex}\label{ex:ng} Let $\gg$ be a perm algebra with basis $\mathcal{B}=\{e_1,e_2\}$ and the product is defined by
 \begin{center}
        \begin{tabular}{r|rr}
          $\cdot$ & $e_1$  & $e_2$   \\
          \hline
           $e_1$ & $e_1$  & $0$ \\
           $e_2$ & $e_2$  & $0$  \\
        \end{tabular}.
        \end{center}
        Let $k_i$, $i=1,2,3,4$ be parameters, define
        $$\left\{
            \begin{array}{l}
             N(e_1)=k_1 e_1+k_2 e_2\\
             N(e_2)=k_3 e_2\\
            \end{array}
            \right., $$
     then $(\gg,\cdot,N)$ is a Nijenhuis perm algebra.
     
     Set
        $$\left\{
            \begin{array}{l}
            \D(e_1)=-e_2\o e_2\\
             \D(e_2)=0\\
            \end{array}
            \right.$$
        and
      $$\left\{
            \begin{array}{l}
             S(e_1)=k_3 e_1+k_4 e_2\\
             S(e_2)=k_3 e_2\\
            \end{array}
            \right., $$
      then $(\gg,\D,S)$ is a Nijenhuis perm coalgebra. Furthermore, $((\gg,\cdot,N),\D,S)$ is a Nijenhuis perm bialgebra.
 \end{ex}

 By \cite[Theorem 4.11]{Hou}, we have
 \begin{lem} \label{lem:jn} Let $(\gg,\cdot)$ be a perm algebra. Suppose that there is a perm algebra structure $\circ$ on its dual space $\gg^*$. Then the triple $(\gg,\cdot,\D)$ is a perm bialgebra if and only if $(\gg,\gg^*,{R_\cdot}^*-{L_\cdot}^*,{R_\cdot}^*,{R_\circ}^*-{L_\circ}^*,{R_\circ}^*)$ is a matched pair of perm algebras.
 \end{lem}

 \begin{thm}\label{thm:hup} Let $(\gg, \cdot, N)$ be a Nijenhuis perm algebra. Suppose that there is a Nijenhuis perm algebra structure $(\gg^*,\circ,S^*)$ on its dual space. Then the triple $((\gg,\cdot,N),\D,S)$ is a Nijenhuis perm bialgebra if and only if $(\gg,\gg^*,N,S^*,{R_\cdot}^*-{L_\cdot}^*,{R_\cdot}^*,{R_\circ}^*-{L_\circ}^*,{R_\circ}^*)$  is a matched pair of Nijenhuis perm algebras.
 \end{thm}

 \begin{proof} $(\Longrightarrow)$ Let $(\gg,\cdot,N,\D,S)$ be a Nijenhuis perm bialgebra, by Definition \ref{de:yy}, we know that $(\gg,\cdot,\D)$ is a perm bialgebra and $S$, $N^*$ are admissible to $(\gg, N)$ and $(\gg^*, S^*)$ respectively. These mean that $(\gg^*,{R_\cdot}^*-{L_\cdot}^*,{R_\cdot}^*,S^*)$ is a representation of $(\gg, N)$ and $(\gg,{R_\circ}^*-{L_\circ}^*,{R_\circ}^*,N)$ is a representation of $(\gg^*, S^*)$. By Lemma \ref{lem:jn}, $(\gg,\gg^*,{R_\cdot}^*-{L_\cdot}^*,{R_\cdot}^*,{R_\circ}^*-{L_\circ}^*,{R_\circ}^*)$ is a matched pair of perm algebras. Therefore, $(\gg,\gg^*,N,S^*,{R_\cdot}^*-{L_\cdot}^*,{R_\cdot}^*,{R_\circ}^*-{L_\circ}^*,{R_\circ}^*)$  is a matched pair of Nijenhuis perm algebras.

 $(\Longleftarrow)$ Let $(\gg,\gg^*,N,S^*,{R_\cdot}^*-{L_\cdot}^*,{R_\cdot}^*,{R_\circ}^*-{L_\circ}^*,{R_\circ}^*)$  be a matched pair of Nijenhuis perm algebras. Then $(\gg,\gg^*,{R_\cdot}^*-{L_\cdot}^*,{R_\cdot}^*,{R_\circ}^*-{L_\circ}^*,{R_\circ}^*)$ is a matched pair of perm algebras and $S$, $N^*$ are admissible to $(\gg, N)$ and $(\gg^*, S^*)$ respectively. Hence, By Lemma \ref{lem:jn} and Definition \ref{de:yy}, $((\gg,\cdot,N),\D,S)$ is a Nijenhuis perm bialgebra.
 \end{proof}

 \begin{thm}\label{thm:pog} Let $(\gg,\cdot,N)$ and $(\gg^*,\circ:=\D^*,S^*)$ be two Nijenhuis perm algebras. Then the following conditions are equivalent:
 \begin{enumerate}[(1)]
 \item $(\gg,\gg^*,N,S^*,{R_\cdot}^*-{L_\cdot}^*,{R_\cdot}^*,{R_\circ}^*-{L_\circ}^*,{R_\circ}^*)$ is a matched pair of Nijenhuis perm algebras $(\gg,\cdot,N)$ and $(\gg^*,\circ,S^*)$.
 \item There is a Manin triple of NF perm algebra associated to $(\gg,\cdot,N)$ and $(\gg^*,\circ,S^*)$.
 \item $((\gg,\cdot,N),\D,S)$ is a Nijenhuis perm bialgebra.
 \end{enumerate}
 \end{thm}

 \begin{proof} It is obvious by Theorem \ref{thm:hu} and Theorem \ref{thm:hup}.
 \end{proof}

 \subsection{Coboundary Nijenhuis perm bialgebra}
 Let $(\gg, \cdot)$ be a perm algebra. For $r\in \gg\o \gg$, in \cite{Hou} $\D_r$ is defined by
 \begin{equation}\mlabel{eq:cop}
 \D(x):=\D_r(x):=xr^{1}\o r^{2}+r^{1}\o xr^{2}-r^{1}\o r^{2}x, \quad \forall x\in \gg.
 \end{equation}
 \begin{defi}\label{de:bh} A Nijenhuis perm bialgebra $((\gg,\cdot,N),\D,S)$ is called {\bf coboundary} if there exists a $r \in \gg \o \gg$ such that Eq.(\ref{eq:cop}) holds.
 \end{defi}

 In \cite{Hou}, the author provided the following results. Let $(\gg,\cdot)$ be a perm algebra, $r \in \gg\o \gg$ and $\D: \gg\lr \gg \o \gg$ be the linear map defined by Eq.(\ref{eq:cop}). Then $\D$ makes $(\gg,\D)$ into a perm coalgebra such that $(\gg,\cdot,\D)$ is a perm bialgebra if and only if for all $x,y\in \gg$
 \begin{eqnarray}
 &&(L(xy)\o \id+\id\o L(xy))(r-\tau(r))=(L(yx)\o \id+\id\o L(yx))(r-\tau(r)),\label{eq:ji}\\
 &&(R(y)\o L(x)+L(y) \o R(x)+\id \o L(yx))(r-\tau(r))\nonumber\\
 &&=(R(y)\o R(x)+\id \o R(yx)+\id \o L(x y))(r-\tau(r)),\label{eq:jp}\\
 &&(L(x) \o L(y))(r-\tau(r))=0,\label{eq:jd}\\
 &&(R(y)\o L(x)+L(y)\o R(x)+\id \o L(yx))(r-\tau(r))\nonumber\\
 &&=(R(y)\o R(x)+L(y)\o L(x)+\id\o L(xy))(r-\tau(r)),\label{eq:jm}\\
 &&(\id\o \id\o(L(x)-R(x)))P(r)=(L(x)\o \id\o \id)J(r),\label{eq:hu}\\
 &&(\id\o \id\o(L(x)-R(x)))Q(r)=(L(x)\o \id\o \id)T(r)+M(r),\label{eq:hs}
 \end{eqnarray}
 where
 \begin{eqnarray*}
 &&P(r)=r_{13}r_{12}-r_{13}r_{23}+r_{23}r_{12}-r_{12}r_{23},\\
 &&Q(r)=r_{13}r_{12}-r_{13}r_{32}+r_{23}r_{12}-r_{12}r_{23},\\
 &&J(r)=r_{12}r_{23}+r_{13}r_{23}-r_{12}r_{13}-r_{23}r_{13},\\
 &&T(r)=r_{12}r_{23}+r_{13}r_{32}-r_{13}r_{12}-r_{32}r_{12},\\
 &&M(r)=(\id\o L(x)\o \id)(r_{32}r_{12}-r_{23}r_{12})+(\id\o R(x)\o \id)(r_{12}r_{23}-r_{12}r_{32}),
 \end{eqnarray*}
 $r_{13}r_{12}=r^1\bar{r}^1\o \bar{r}^2\o r^2$ and similarly for others.

 Let $(\gg,N)$ be an $S$-admissible Nijenhuis perm algebra. In order for $((\gg,\cdot,N),\D,S)$ to be a Nijenhuis perm bialgebra, we only need to further require that $(\gg^*,{\D}^*,S^*)$ is an $N^*$-admissible Nijenhuis perm algebra, that is, $(\mathfrak{\gg},\D,S)$ is a Nijenhuis perm coalgebra and Eq.(\ref{eq:jy}) and (\ref{eq:jj}) hold.

 \begin{pro}\label{pro:nh} Let $(\gg,N)$ be an $S$-admissible Nijenhuis perm algebra and $r\in \gg \o \gg$. Define a linear map $\D: \gg\lr \gg \o \gg$  by Eq.(\ref{eq:cop}). Then the following conclusions hold.
 \begin{enumerate}[(1)]
 \item \label{it:7} Eq.(\ref{eq:gd}) holds if and only if for all $x\in \gg$,
 \begin{eqnarray}\label{eq:mg}
 &&(\id \o S\circ L(x)-\id \o S\circ R(x)-\id\o L(S(x))+\id\o R(S(x)))(S\o \id -\id \o N)(r)\nonumber\\
 &&+(L(S(x))\o \id-S\circ L(x)\o \id)(N\o \id-\id \o S)(r)=0.
 \end{eqnarray}
 \item \label{it:44} Eq.(\ref{eq:jj}) holds if and only if for all $x\in \gg$,
 \begin{eqnarray}\label{eq:jk}
 &&(\id \o L(N(x))+L(N(x))\o \id +\id\o S\circ L(x) -\id \o S\circ R(x)-\id\o R(N(x))-N\circ L(x)\o \id)\nonumber\\
 &&(\id\o S-N\o \id)(r)+(\id\o R(x)-\id\o L(x))(\id\o S^2- N^2\o \id )(r)=0.
 \end{eqnarray}
 \item \label{it:65} Eq.(\ref{eq:jy}) holds if and only if for all $x\in \gg$,
 \begin{eqnarray}\label{eq:jkw}
 &&(\id \o L(N(x))+L(N(x))\o \id +S\circ L(x)\o \id -\id\o R(N(x))-\id \o N(L(x))+\id \o N(R(x))) \nonumber\\
 &&(S\o \id -\id\o N)(r)+(L(x)\o \id)(\id\o N^2-S^2\o \id )(r)=0.
 \end{eqnarray}
 \end{enumerate}
 \end{pro}

 \begin{proof} \ref{it:7} For all $x\in \gg$, by Eq.(\ref{eq:cop}), we have
 \begin{eqnarray*}
 &&(S \o S)\D(x)= S(x r^1)\o S(r^2)+S(r^1)\o S(a r^2)-S(r^1)\o S(r^2 a),\\
 &&\D(S^2(x))=S^2(x)r^1\o r^2+r^1\o S^2(x)r^2-r^1\o r^2 S^2(x),\\
 &&(S\o\id)\D(S(x))=S(S(x)r^1)\o r^2+S(r^1) \o S(x)r^2-S(r^1) \o r^2 S(x)\\
 &&\hspace{27mm}\stackrel {(\ref{eq:sh})}{=}S(x N(r^1)) \o r^2+S^2(x)r^1\o r^2-S(x)N(r^1)\o r^2+S(r^1) \o S(x)r^2\\
 &&\hspace{32mm}-S(r^1)\o r^2 S(x),\\
 &&(\id\o S)\D(S(x))=S(x)r^1\o S(r^2)+r^1\o S(S(x)r^2)-r^1\o S(r^2 S(x))\\
 &&\hspace{25mm}\stackrel {(\ref{eq:sh})(\ref{eq:sfh})}{=}S(x)r^1\o S(r^2)+r^1\o S(x N(r^2))+r^1\o S^2(x)r^2-r^1\o S(x)N(r^2)\\
 &&\hspace{35mm}-r^1\o S(N(r^2)x)-r^1\o r^2 S^2(x)+r^1\o N(r^2)S(x).
 \end{eqnarray*}
 Then Eq.(\ref{eq:gd}) holds if and only if Eq.(\ref{eq:mg}) holds.
 
 \ref{it:44} and \ref{it:65}: Similar to the proof of \ref{it:7}. 
 \end{proof}

 \begin{lem}\label{lem:pq}
 \begin{enumerate}[(1)]
 \item If $(S\o \id-\id\o N)(r)=0$, then $(\id\o N^2- S^2\o \id)(r)=0$.
 \item If $(N\o \id-\id\o S)(r)=0$, then $(N^2\o \id-\id\o S^2)(r)=0$.
 \item If $r$ is symmetric, then $(S\o \id-\id\o N)(r)=0$ if and only if $(N\o \id-\id\o S)(r)=0$.
 \end{enumerate}
 \end{lem}

 \begin{proof}
 It can be proved directly.
 \end{proof}

 \begin{thm}\label{thm:ggg} Let $(\gg,N)$ be an $S$-admissible Nijenhuis perm algebra and $r\in \gg \o \gg$. Define a linear map $\D$ by Eq.(\ref{eq:cop}). Then $((\gg,\cdot,N),\D,S)$ is a Nijenhuis perm bialgebra if and only if Eqs.(\ref{eq:ji})-(\ref{eq:jkw}) hold.
 \end{thm}

 \begin{proof} $((\gg,\cdot,N),\D,S)$ is a Nijenhuis perm bialgebra if and only if $(\gg^*,\D^*,S^*)$ is an $N^*$-admissible Nijenhuis perm algebra. By Proposition \ref{pro:nh}, the latter holds if and only if Eqs.(\ref{eq:ji})-(\ref{eq:jkw}) hold.
 \end{proof}

 \begin{cor}\label{cor:gm} Let $(\gg,N)$ be an $S$-admissible Nijenhuis perm algebra and $r \in \gg\o \gg$ be symmetric. Define a linear map $\D: \gg\lr \gg\o \gg$ by Eq.(\ref{eq:cop}) . Then $((\gg,\cdot,N),\D,S)$ is a Nijenhuis perm bialgebra if the following equations hold:
 \begin{eqnarray}
 &P(r)=r_{13}r_{12}-r_{13}r_{23}+r_{23}r_{12}-r_{12}r_{23}=0,&\label{eq:uj}\\
 &(S\o \id-\id\o N)(r)=0.&\label{eq:cg}
 \end{eqnarray}
 \end{cor}

 \begin{defi}\label{de:nj} Let $(\gg,N)$ be a Nijenhuis perm algebra, $S:\gg\lr \gg$ a linear map. Then Eq.(\ref{eq:uj}) together with Eq.(\ref{eq:cg}) is called an {\bf $S$-admissible perm Yang-Baxter equation in $(\gg, N)$}.
 \end{defi}

 \begin{pro}\label{pro:gth} Let $(\gg,N)$ be an $S$-admissible Nijenhuis perm algebra and $r \in \gg\o \gg$ be a symmetric solution of the $S$-admissible perm Yang-Baxter equation in $(\gg, N)$. Then $((\gg,\cdot,N),\D,S)$ is a Nijenhuis perm bialgebra, where the linear map $\D=\D_r$ is defined by Eq.(\ref{eq:cop}). We call this Nijenhuis perm bialgebra {\bf quasitriangular}.
 \end{pro}

 \begin{proof}
 It is obvious by Corollary \ref{cor:gm} and Definition \ref{de:nj}.
 \end{proof}

 For a vector space $\gg$, by isomorphism $\gg\o \gg\cong Hom(\gg^*, \gg)$ and $r\in V\o V$, define a linear map $ r^\sharp: V^*\lr V$ by
 \begin{eqnarray*}\label{eq:nm}
 r^\sharp(u^*)=\langle u^*,r^1\rangle r^2, \forall u^*\in V^*
 \end{eqnarray*}
 We say that $r\in V\o V$ is {\bf nondegenerate} if the linear map $r^\sharp$ is a bijection.

 \begin{thm}\label{thm:x} Let $(\gg,N)$ be a Nijenhuis perm algebra and $r\in \gg\o \gg$ be symmetric. Let $S: \gg\lr \gg$ be a linear map. Then $r$ is a solution of the $S$-admissible perm Yang-Baxter equation in $(\gg,N)$ if and only if $r^\sharp$ satisfies
 \begin{eqnarray}
 &r^\sharp(x^*)r^\sharp(y^*)=r^\sharp(y^* R^*(r^\sharp(x^*))-L^*(r^\sharp(x^*))y^*+x^* R^*(r^\sharp(y^*))),\forall x^*, y^*\in \gg^* &\label{eq:lq}\\
 &N\ci r^\sharp=r^\sharp\ci S^*.&\label{eq:lu}
 \end{eqnarray}
 \end{thm}

 \begin{proof} By \cite[Proposition 4.19]{Hou}, $r$ is a solution of perm Yang-Baxter equation in $\gg$ if and only if Eqs.(\ref{eq:lq}) holds. Moreover, for any $x^* \in \gg^*$, we have
 \begin{eqnarray*}
 r^\sharp(S^*(x^*))\stackrel{}{=}\langle S^*(x^*),r^1\rangle r^2=\langle x^*,S(r^1)\rangle r^2,
 \end{eqnarray*}
 and
 \begin{eqnarray*}
 N(r^\sharp(x^*))\stackrel{}{=}\langle x^*,r^1\rangle N(r^2).
 \end{eqnarray*}
 Hence, Eq.(\ref{eq:lu}) holds if and only if Eq.(\ref{eq:cg}) holds. These finish the proof.
 \end{proof}

 \subsection{$\mathcal{O}$-operators on Nijenhuis perm algebra}
 \begin{defi}\label{de:hc} Let $(\gg,N)$ be a Nijenhuis perm algebra, $(V,\ell,r)$ be a representation of $\gg$ and $\a: V\lr V$ be a linear map. A linear map $T:V\lr \gg$ is called a {\bf weak $\mathcal{O}$-operator associated to $(V,\ell,r)$ and $\a$} if $T$ satisfies
 \begin{eqnarray}
 &T(x)T(y)=T(\ell(T(x))y+xr(T(y))), \forall x, y \in V,&\label{eq:lz}\\
 &N\circ T=T\circ \a.&\label{eq:lm}
 \end{eqnarray}
 If $(V,\ell,r,\a)$ is a representation of $(\gg,N)$, then $T$ is called an {\bf $\mathcal{O}$-operator associated to $(V,\ell,r,\a)$}.
 \end{defi}

 Theorem \ref{thm:x} can be rewritten in terms of $\mathcal{O}$-operators as follows.

 \begin{cor}\label{cor:bbo} Let $(\gg,N)$ be a Nijenhuis perm algebra, $r\in \gg\o \gg$ be symmetric and $S:\gg\lr \gg$ be a linear map. Then $r$ is a solution of the $S$-admissible perm Yang-Baxter equation in $(\gg,N)$ if and only if $r^\sharp$ is a weak $\mathcal{O}$-operator associated to $(\gg^*,R^*-L^*,R^*)$ and $S^*$. In addition, if $(\gg,N)$ is an $S$-admissible Nijenhuis perm algebra, then $r$ is a solution of the $S$-admissible perm Yang-Baxter equation in $(\gg,N)$ if and only if $r^\sharp$ is an $\mathcal{O}$-operator associated to $(\gg^*,R^*-L^*,R^*,S^*)$.
 \end{cor}

 In what follows, we will prove that $\mathcal{O}$-operators can provide solutions of the admissible perm Yang-Baxter equation in semi-direct product Nijenhuis perm algebras and then produce Nijenhuis perm bialgebras.

 \begin{lem}\label{lem:ri} Let $\gg$ be a perm algebra, $(V,\ell,r)$ be a representation of $\gg$ and $T:V\lr \gg$ be a linear map which is identified as an element in $(\gg\ltimes_{r^*-\ell^*,r^*}V^*)\o (\gg\ltimes_{r^*-\ell^*,r^*}V^*)$ through $Hom(V,\gg)\cong\gg\o V^*\subseteq (\gg\ltimes_{r^*-\ell^*,r^*}V^*)\o (\gg\ltimes_{r^*-\ell^*,r^*}V^*)$. Then $r=T+\tau(T)$ is a symmetric solution of the perm Yang-Baxter equation in the semi-direct product perm algebra $\gg\ltimes_{r^*-\ell^*,r^*}V^*$ if and only if $T$ is an $\mathcal{O}$-operator of $(\gg,\cdot)$ associated to $(V,\ell,r)$.
 \end{lem}

 \begin{proof} Let $\{e_1,e_2,\cdots, e_n\}$ be a basis of $V$ and $\{e_1^*,e_2^*,\cdots, e_n^*\}$ be the dual basis. Then
 \begin{eqnarray*}
 r=\sum_{i=1}^n (T(e_i)\o e_i^*+e_i^*\o T(e_i)).
 \end{eqnarray*} 
 Then, we get
 \begin{eqnarray*}
 &&\hspace{-5mm}r_{13}r_{12}-r_{13}r_{23}+r_{23}r_{12}-r_{12}r_{23}\\
 &&=\sum_{1\leq i,j\leq n}((T(e_i)T(e_j)-T(e_ir(T(e_j)))-T(\ell(T(e_i))e_j))\o e_j^*\o e_i^*\\
 &&\hspace{5mm}-e_i^*\o (T(e_i)T(e_j)-T(e_ir(T(e_j)))-T(\ell(T(e_i))e_j))\o e_j^*\\
 &&\hspace{5mm}+e_j^*\o (T(e_i)T(e_j)-T(e_ir(T(e_j)))-T(\ell(T(e_i))e_j))\o e_i^*\\
 &&\hspace{5mm}-e_i^*\o e_j^*\o(T(e_i)T(e_j)-T(e_ir(T(e_j)))-T(\ell(T(e_i))e_j))).
 \end{eqnarray*}
 So $r$ is a solution of the perm Yang-Baxter equation in the semi-direct product perm
 algebra $\gg\ltimes_{r^*-\ell^*,r^*}V^*$ if and only if $T$ is an $\mathcal{O}$-operator associated to $(V,\ell,r)$.
  \end{proof}

 \begin{thm}\label{thm:fv} Let $(\gg,N)$ be a Nijenhuis perm algebra, $(V,\ell,r)$ be a representation of $(\gg,\cdot)$, $S : \gg\lr \gg$ and $\a, \beta:V\lr V$ be linear maps. Then the following conditions are equivalent.
 \begin{enumerate}[(1)]
 \item \label{it:e2} There is a Nijenhuis perm algebra $(\gg \ltimes_{\ell,r} V, N+\a)$ such that the linear map $S+\beta$ on $\gg\oplus V$ is admissible to $(\gg \ltimes_{\ell,r} V, N+\a)$.
 \item \label{it:e3} There is a Nijenhuis perm algebra $(\gg \ltimes_{r^*-\ell^*,r^*} V^*, N+\beta^*)$ such that the linear map $S+\a^*$ on $\gg\oplus V^*$ is admissible to $(\gg \ltimes_{r^*-\ell^*,r^*} V^*, N+\beta^*)$.
 \item \label{it:e1} The following conditions are satisfied:
 \begin{enumerate}[(a)]
 \item \label{it:a1} $(V,\ell,r,\a)$ is a representation of $(\gg, N)$;
 \item \label{it:a2} $S$ is admissible to $(\gg, N)$;
 \item \label{it:a3} $\beta$ is admissible to $(\gg,  N)$ on $(V, \ell,r)$;
 \item \label{it:a4} For all $x\in \gg$ and $u\in V$, the following equation holds:
 \begin{eqnarray}
 &&\beta(\ell(x)\a(u))+\ell(S^2(x))u=\ell(S(x))\a(u)+\beta(\ell(S(x))u),\label{eq:hjl} \\
 &&\beta(\a(u)r(x))+ur(S^2(x))=\a(u)r(S(x))+\beta(ur(S(x))).\label{eq:hg}
 \end{eqnarray}
 \end{enumerate}
 \end{enumerate}
 \end{thm}

 \begin{proof} $\ref{it:e2}\Longleftrightarrow\ref{it:e1}$ By Proposition \ref{pro:B}, we have $(\gg \ltimes_{\ell,r} V, N+\a)$ is a Nijenhuis perm algebra if and only if $(V,\ell,r,\a)$ is a representation of $(\gg,N)$. Let $x,y\in \gg$ and $u,v\in V$, then we have
 \begin{eqnarray*}
 &&(S+\beta)((N+\a)(x+u))(y+v))=S(N(x)y)+\beta(\ell(N(x))v+\a(u)r(y)),\\
 &&(x+u)(S+\beta)^2(y+v)=xS^2(y)+\ell(x)\beta^2(v)+ur(S^2(y)),\\
 &&((N+\a)(x+u))((S+\beta)(y+v))=N(x)S(y)+\ell(N(x))\beta(v)+\a(u)r(S(y)),\\
 &&(S+\beta)((x+u)((S+\beta)(y+v)))=S(xS(y))+\beta(\ell(x)\beta(v)+ur(S(y))).
 \end{eqnarray*}
 We obtain that
 \begin{eqnarray}\label{eq:cbv}
 &&S(N(x)y)+\beta(\ell(N(x))v)+\beta(\a(u)r(y))+xS^2(y)+\ell(x)\beta^2(v)+ur(S^2(y))-N(x)S(y)\nonumber \\
 &&-\ell(N(x))\beta(v)-\a(u)r(S(y))-S(xS(y))-\beta(\ell(x)\beta(v))-\beta(ur(S(y)))=0.
 \end{eqnarray}
 Then, Eq.(\ref{eq:cbv}) holds if and only if Eq.(\ref{eq:sfh}) (corresponding $u=v=0$), Eq.(\ref{eq:fb}) (corresponding $u=y=0$) and Eq.(\ref{eq:hg}) (corresponding $x=v=0$) hold. Similarly, Eqs.(\ref{eq:fs}), (\ref{eq:sh}) and (\ref{eq:hjl}) hold. Hence, \ref{it:e2} holds if and only if \ref{it:e1} holds.

 $\ref{it:e3}\Longleftrightarrow\ref{it:e1}$ In \ref{it:e2}, take
 \begin{eqnarray*}
 V=V^*, \rho=\rho^*, \ell=r^*-\ell^*, r=r^*, \a=\beta^*, \beta=\a^*.
 \end{eqnarray*}
 Then from the above equivalence between \ref{it:e2} and \ref{it:e1}, we have \ref{it:e3} holds if and only if \ref{it:a1}-\ref{it:a3} hold and for all $x\in \gg, v^*\in V^*$, the following equation holds:
 \begin{eqnarray}
 &&\a^*(\beta^*(v^*)r^*(x))+v^*r^*(S^2(x))-\beta^*(v^*)r^*(S(x))-\a^*(v^*r^*(S(x)))=0, \label{eq:cse}\\
 &&\a^*((r^*-\ell^*)(x)\beta^*(v^*))+(r^*-\ell^*)(S^2(x))v^*-(r^*-\ell^*)(S(x))\beta^*(v^*)\nonumber\\
 &&\hspace{67mm}-\a^*((r^*-\ell^*)(S(x))v^*)=0.\label{eq:pp}
 \end{eqnarray}
 For all $x \in \gg$, $u\in V$ and $v^* \in V^*$, 
 one can get that
 \begin{eqnarray*}
 &&\hspace{-16mm}\langle\a^*(\beta^*(v^*) r^*(x))+v^* r^*(S^2(x))-\beta^*(v^*) r^*(S(x))-\a^*(v^* r^*(S(x))),u\rangle\\
 &=&\langle v^*,\beta(\a(u) r(x))+u r(S^2(x))-\beta(u r(S(x))-\a(u) r(S(x))\rangle.
 \end{eqnarray*}
 We can obtain Eq.(\ref{eq:cse}) holds if and only if Eq.(\ref{eq:hg}) holds. Similarly, Eq.(\ref{eq:pp}) holds if and only if Eq.(\ref{eq:hjl}) holds. Therefore, \ref{it:e3} holds if and only if \ref{it:e1} holds.
 \end{proof}

 \begin{thm}\label{thm:hnv} Let $(\gg,N)$ be a Nijenhuis perm algebra, $(V,\ell,r)$ be a representation of $(\gg,\cdot)$, $(V^*, r^*-\ell^*,r^*,\beta^*)$ be a representation of $(\gg, N)$, $S:\gg\lr \gg$, $\a:V\lr V$ and $T:V\lr \gg$ be linear maps.
 \begin{enumerate}[(1)]
 \item \label{it:xn} $r=T+\tau(T)$ is a symmetric solution of the $(S+\a^*)$-admissible perm Yang-Baxter equation in the Nijenhuis perm algebra $(\gg\ltimes_{r^*-\ell^*,r^*}V^*, N+\beta^*)$ if and only if $T$ is a weak $\mathcal{O}$-operator associated to $(V,\ell,r)$ and $\a$, and satisfies $T\circ \beta=S\circ T$.
 \item \label{it:r6} Assume that $(V,\ell,r,\a)$ is a representation of $(\gg, N)$. If $T$ is a weak $\mathcal{O}$-operator associated to $(V,\ell,r,\a)$ and $T\circ \beta=S\circ T$, then $r=T+\tau(T)$ is a symmetric solution of the $(S+\a^*)$-admissible perm Yang-Baxter equation in the Nijenhuis perm algebra $(\gg\ltimes_{r^*-\ell^*,r^*}V^*, N+\beta^*)$. In addition, if $(\gg, N)$ is  $S$-admissible and Eqs.(\ref{eq:hjl}) and (\ref{eq:hg}) hold such that the Nijenhuis perm algebra $(\gg\ltimes_{r^*-\ell^*,r^*}V^*, N+\beta^*)$ is $(S+\a^*)$-admissible, then there is a Nijenhuis perm bialgebra $(\gg\ltimes_{r^*-\ell^*,r^*}V^*, N+\beta^*,\D, S+\a^*)$, where the linear map $\D=\D_r$ is defined by Eq.(\ref{eq:cop}) with $r=T+\tau(T)$.
 \end{enumerate}
 \end{thm}

 \begin{proof} It is similar to \cite[Theorem 5.17]{MLo}.  \end{proof}

 \subsection{Classification of 2-dimensional quasitriangular noncommutative perm bialgebras}
 \begin{thm}\label{thm:gu} There are four kinds of noncommutative perm algebras of dimension 2 (see \cite[Example 2.3]{Hou}). We now give all quasitriangular perm bialgebraic structures on these 2-dimension perm algebras, where $\l, \nu, \kappa$ are parameters, see Items \ref{it:1}-\ref{it:4}.
 \begin{enumerate}[(a)]
 \item \label{it:1}For perm algebra
 \begin{align*}
     \begin{tabular}{r|rr}
          $\cdot$ & $e_1$  & $e_2$  \\
          \hline
           $e_1$ & $e_1$  & $0$  \\
           $e_2$ & $e_2$  & $0$ \\
     \end{tabular},
     \end{align*}
\begin{center}
 \begin{tabular}[t]{c|c}
 \hline\hline
 \bf{perm $r$-matrices}  & \bf{perm coalgebra induced by $r$-matrices}\\
 \hline
 $\makecell[c]{r=\lambda e_1 \o e_1+\kappa e_1\o e_2\\\quad+\kappa e_2\o e_1+\nu e_2\o e_2}~~~~(\kappa^2=\l\nu)$
 & $\left\{
 \begin{array}{ll}
 \d_r(e_1)=\l e_1\o e_1-\nu e_2\o e_2&\\
 \d_r(e_2)=\l(e_2\o e_1+e_1\o e_2)+2\kappa e_2\o e_2 &\\
 \end{array}
 \right.$. \\
   \hline
   \end{tabular}
   \end{center}
    \item \label{it:2}For perm algebra
 \begin{align*}
     \begin{tabular}{c|cc}
          $\cdot$ & $e_1$  & $e_2$  \\
          \hline
           $e_1$ & $e_1+e_2$  & $0$  \\
           $e_2$ & $e_2$  & $0$ \\
     \end{tabular},
     \end{align*}
\begin{center}
 \begin{tabular}[t]{c|c}
 \hline\hline
 \bf{perm $r$-matrices}  & \bf{perm coalgebra induced by $r$-matrices}\\
 \hline
 $r=\l e_2 \o e_2$
 & $\left\{
 \begin{array}{ll}
 \d_r(e_1)=-\l e_2\o e_2&\\
 \d_r(e_2)=0&\\
 \end{array}
 \right..$ \\
   \hline
   \end{tabular}
   \end{center}
\item \label{it:2}For perm algebra
  \begin{align*}
     \begin{tabular}{r|rr}
          $\cdot$ & $e_1$  & $e_2$  \\
          \hline
           $e_1$ & $0$  & $e_1$  \\
           $e_2$ & $0$  & $e_2$ \\
     \end{tabular},
     \end{align*}
\begin{center}
 \begin{tabular}[t]{c|c}
 \hline\hline
 \bf{perm $r$-matrices}  & \bf{perm coalgebra induced by $r$-matrices}\\
 \hline
 $r=\lambda e_1 \o e_1$
 & $\left\{
 \begin{array}{ll}
 \d_r(e_1)=0&\\
 \d_r(e_2)=\l e_1\o e_1 &\\
 \end{array}
 \right.$ \\
 \hline
 $r=\l e_2\o e_2$
 & $\left\{
 \begin{array}{ll}
  \d_r(e_1)=\l(e_1\o e_2+e_2\o e_1)&\\
  \d_r(e_2)=\l e_2\o e_2&\\
  \end{array}
  \right..$\\
   \hline
   \end{tabular}
   \end{center}
    \item \label{it:4}For perm algebra
 \begin{align*}
     \begin{tabular}{c|cc}
          $\cdot$ & $e_1$  & $e_2$  \\
          \hline
           $e_1$ & $0$  &  $e_1$  \\
           $e_2$ & $0$  & $e_1+e_2$ \\
     \end{tabular},
     \end{align*}
\begin{center}
 \begin{tabular}[t]{c|c}
 \hline\hline
 \bf{perm $r$-matrices}  & \bf{perm coalgebra induced by $r$-matrices}\\
 \hline
 $r=\lambda e_1 \o e_1$
 & $\left\{
 \begin{array}{ll}
 \d_r(e_1)=0&\\
 \d_r(e_2)=- \l e_1\o e_1&\\
 \end{array}
 \right..$ \\
   \hline
   \end{tabular}
   \end{center}
   \end{enumerate}
 \end{thm}
 
 \begin{proof} It can be proved by a direct computation.
 \end{proof}

 \begin{rmk} \cite[Example 4.17]{Hou} corresponds to the case of $\l=0, \nu=1, \kappa=0$ in Item \ref{it:1}. We note that there is missing a negative sign in \cite[Example 4.17]{Hou} and $\D(e_1)$ should be $-e_2\o e_2$.
 \end{rmk}

\section{A method to obtain a \N perm algebra}\label{se:n} In this section, we give a method to construct a \N operator on a perm algebra. The method is strongly dependent on the perm Yang-Baxter equation, so it is essentially different to the case of associative algebras in \cite{MLo}.

 \subsection{Symplectic perm algebras from dual quasitriangular perm bialgebras}
 \begin{pro} \label{de:qt} Let $(\gg, \cdot)$ be a perm algebra. If $r\in \gg\o \gg$ is a symmetric solution of the {\bf classical perm Yang-Baxter equation} in $(\gg, \cdot)$, i.e., Eq.(\ref{eq:uj}). 
 Then $(\gg, \cdot, \D_r)$ is a perm bialgebra, where $\D_r$ is defined by Eq.(\ref{eq:cop}). In this case, we call this perm bialgebra {\bf quasitriangular} and denoted by $(\gg, \cdot, r, \D_r)$.
 \end{pro}

 \begin{proof}
  It is direct by \cite[Corollary 4.15]{Hou}.
 \end{proof}

 \begin{pro}\mlabel{pro:qt}
 Let $(\gg, \cdot)$ be a perm algebra and $r\in \gg\o \gg$ be a symmetric element. Then the quadruple $(\gg, \cdot, r, \D_r)$ is a quasitriangular perm bialgebra, where $\D_r$ is defined by Eq.(\ref{eq:cop}), if and only if
 \begin{eqnarray}
 &(\D_r\o \id)(r)=r^1\o \br^1\o r^2\br^2&\mlabel{eq:qt1}
 \end{eqnarray}
 or
 \begin{eqnarray}
 &(\id\o \D_r)(r)=r^1 \br^1\o r^2\o \br^2&\mlabel{eq:qt2}
 \end{eqnarray}
 holds.
 \end{pro}

 \begin{proof} By Eq.(\ref{eq:cop}), we have
 \begin{eqnarray*}\mlabel{eq:qt1-1}
 (\D_r\o \id)(r)
 &\stackrel{ }=&r^1 \br^1\o \br^2\o r^2+\br^1\o r^1 \br^2\o r^2-\br^1\o \br^2 r^1\o r^2.
 \end{eqnarray*}
 Then Eq.(\ref{eq:uj}) $\Leftrightarrow$ Eq.(\ref{eq:qt1}).
 By the symmetry of $r$, we obtain that Eq.(\ref{eq:qt2}) $\Leftrightarrow$  Eq.(\ref{eq:uj}).
 \end{proof}

 \begin{defi}\mlabel{de:ccybe} Let $(\gg, \D)$ be a perm coalgebra and $\om\in (\gg\o \gg)^*$ be a bilinear form. The equation for all $x,y,z\in \gg$
 \begin{eqnarray}\mlabel{eq:ccybe}
 \om(x_{(1)},z)\om(x_{(2)}, y)-\om(x, z_{(1)})\om(y, z_{(2)})+\om(y_{(1)},z)\om(x,y_{(2)})-\om(x,y_{(1)})\om(y_{(2)},z)=0,
 \end{eqnarray}
 is called a {\bf classical co-perm Yang-Baxter equation} in $(\gg, \D)$.
 \end{defi}

 \begin{pro}
  \begin{enumerate}[(1)]
   \item Let $(\gg, \D)$ be a perm coalgebra and $\om\in (\gg\o \gg)^*$ be a solution of a classical co-perm Yang-Baxter equation in $(\gg, \D)$. Then $\om^*\in \gg^*\o \gg^*$ defined by $\om^*(x, y)=\om(x, y)$ is a solution of a classical perm Yang-Baxter equation in $(\gg^*, \D^*)$.
   \item Let $(\gg, \cdot)$ be a finite-dimensional perm algebra and $r\in \gg\o \gg$ be a solution of a classical perm Yang-Baxter equation in $(\gg, \cdot)$. Then $r^*\in (\gg^*\o \gg^*)^*$ defined by $r^*(a^*, b^*)=\langle a^*, r^1\rangle \langle b^*, r^2\rangle$, where $a^*, b^*\in \gg^*$, is a solution of a classical co-perm Yang-Baxter equation in $(\gg^*, \cdot^*)$.
  \end{enumerate}
 \end{pro}

 \begin{proof} It is direct by Proposition \ref{pro:de:cl}.
 \end{proof}

 \begin{thm}\mlabel{de:cqt} Let $(\gg, \D)$ be a perm coalgebra and $\om\in (\gg\o \gg)^*$ be a symmetric (in the sense of $\om(x, y)=\om(y, x)$) solution of the classical co-perm Yang-Baxter equation in $(\gg, \D)$.
 Then $(\gg, \cdot_{\om}, \D)$ is a perm bialgebra, where the multiplication on $\gg$ is defined by
 \begin{eqnarray}\mlabel{eq:p}
 x\cdot_{\om}y=x_{(1)}\om(x_{(2)}, y)+y_{(1)}\om(x, y_{(2)})-y_{(2)}\om(x, y_{(1)}),\quad \forall~x, y\in \gg.
 \end{eqnarray}
 This bialgebra is called a {\bf dual quasitriangular perm bialgebra} denoted by $(\gg, \D, \om, \cdot_{\om})$.
 \end{thm}

 \begin{proof} Straightforward.
 \end{proof}

 \begin{pro}\mlabel{pro:cqt} Let $(\gg, \D)$ be a perm coalgebra and $\om\in (\gg\o \gg)^*$ be symmetric. Then the quadruple $(\gg, \D, \om, \cdot_{\om})$ is a dual quasitriangular perm bialgebra, where $\cdot_{\om}$ is defined by Eq.(\ref{eq:p}), if and only if
 \begin{eqnarray}\mlabel{eq:cqt1}
 \om(x\cdot_{\om} y, z)=\om(x, z_{(1)})\om(y, z_{(2)})
 \end{eqnarray}
 or
 \begin{eqnarray}\mlabel{eq:cqt2}
 \om(x, y\cdot_{\om}z)=\om(x_{(1)}, y)\om(x_{(2)}, z)
 \end{eqnarray}
 holds, where $x, y, z\in \gg$.
 \end{pro}

 \begin{proof} It is direct by Eqs.(\ref{eq:p}) and (\ref{eq:ccybe}).
 \end{proof}

 \begin{defi}
  Let $(\gg, \cdot)$ be a perm algebra and $\om\in (\gg\o \gg)^*$ be symmetric. If for all $x, y, z\in \gg$, the equation below holds:
 \begin{eqnarray}\mlabel{eq:symp}
 \om(x y, z)+\om(x z,y)=\om(z x, y)+\om(z y,x).
 \end{eqnarray}
 Then $(\gg, \cdot)$ is a {\bf symplectic perm algebra} denoted by $(\gg, \cdot, \om)$.
 \end{defi}

 \begin{pro}\mlabel{pro:de:sym} Let $(\gg, \D)$ be a perm coalgebra and $\om\in (\gg\o \gg)^*$ be symmetric. If $(\gg, \D, \om, \cdot_{\om})$ is a dual quasitriangular perm bialgebra, where $\cdot_{\om}$ is defined by Eq.(\ref{eq:p}), then $(\gg, \cdot_\om, \om)$ is a symplectic perm algebra.
 \end{pro}

 \begin{proof} For all $x, y, z\in \gg$, we have
 \begin{eqnarray*}
 &&\hspace{-13mm}\om(x\cdot_{\om} y, z)+\om(x\cdot_{\om} z,y)-\om(z\cdot_{\om} x, y)-\om(z\cdot_{\om} y,x)\\
 &\stackrel{(\ref{eq:cqt1})}{=}&\om(x, z_{(1)})\om(y, z_{(2)})+\om(x,y_{(1)})\om(z, y_{(2)})-\om(z,y_{(1)})\om(x, y_{(2)})-\om(z,x_{(1)})\om(y,x_{(2)})\\
 &\stackrel{(\ref{eq:ccybe})}{=}&0,
 \end{eqnarray*}
 as desired.
 \end{proof}

\subsection{\N operators from symplectic perm algebras}

 \begin{thm}\mlabel{thm:ln}
  Let $(\gg, \cdot, \om)$ be a symplectic perm algebra and $r\in \gg\o \gg$ be an element. If $(\gg, \cdot, r, \D_r)$ is a quasitriangular perm bialgebra together with the comultiplication $\D_r$ defined in Eq.(\ref{eq:cop}) and further $(\gg, \D_r, \om, \cdot_{\om})$ is a dual quasitriangular perm bialgebra together with the multiplication $\cdot_{\om}$ given in Eq.(\ref{eq:p}).
 Then $(\gg, N)$ is a \N perm algebra, where $N:\gg\lr \gg$ is given by
 \begin{eqnarray}
 &N(x)=\om(x, r^1)r^2,~~\forall~x\in \gg.&\mlabel{eq:ys}
 \end{eqnarray}
 \end{thm}

 \begin{proof} By Eqs.(\ref{eq:ccybe}) and (\ref{eq:cop}), for all $x, y, z\in \gg$, one has
 \begin{eqnarray*}
 &&\hspace{-10mm}0=\om(x r^1, z)\om(r^2, y)+\om(r^1, z)\om(x r^2, y)
 -\om(r^1,z)\om(r^2 x ,y)\\
 &&-\om(x, z r^1)\om(y, r^2)-\om(x, r^1)\om(y,z r^2)+\om(x,r^1)\om(y, r^2 z)\\
 &&+\om(y r^1, z)\om(r^2, x)+\om(r^1, z)\om(x, y r^2)-\om(r^1,z)\om(x ,r^2 y)\\
 &&-\om(x, y r^1)\om(z, r^2)-\om(x, r^1)\om(y r^2,z)+\om(x,r^1)\om(r^2 y,z)\\
 &&\hspace{-7mm}=\om(x r^1, z)\om(r^2, y)+\om(r^1, z)\om(x r^2, y)
 -\om(r^1,z)\om(r^2 x ,y)\\
 &&-\om(x, z r^1)\om(y, r^2)-\om(x, r^1)\om(y,z r^2)+\om(x,r^1)\om(y, r^2 z)\\
 &&-\om(r^1,z)\om(x ,r^2 y)+\om(x,r^1)\om(r^2 y,z).
 \end{eqnarray*}
 Since, further, $(\gg, \cdot, \om)$ is a symplectic perm algebra, we obtain
 \begin{eqnarray}
 &\om(r^2, y)\om(z x, r^1)-\om(r^2, y)\om(x z, r^1)-\om(r^1,z)\om(x y, r^2)+\om(x,r^1)\om(z y, r^2)=0.&\mlabel{eq:cqt-sym}
 \end{eqnarray}
 Now we can check that $N$ defined in Eq.(\ref{eq:ys}) is a Nijenhuis operator on $(\gg, \cdot)$.
 \begin{eqnarray*}
 &&\hspace{-15mm}N(x)N(y)-N(N(x)y+xN(y)-N(x y))\\
 &\stackrel{(\ref{eq:ys})}{=}&\om(x, r^1)\om(y, \bar{r}^1)r^2\bar{r}^2
 -\om(x, r^1)\om(r^2 y, \bar{r}^1)\bar{r}^2-\om(y, r^1)\om(xr^2, \bar{r}^1)\bar{r}^2\\
 &&+\om(xy, r^1)\om(r^2, \bar{r}^1)\bar{r}^2\\
 &\stackrel{(\ref{eq:cqt-sym})}{=}&\om(x, r^1)\om(y, \bar{r}^1)r^2 \bar{r}^2
 -\om(x, r^1)\om(r^2 y, \bar{r}^1)\bar{r}^2-\om(y, r^1)\om(xr^2, \bar{r}^1)\bar{r}^2\\
 &&-\om(r^2, y)\om(x\br^1, r^1)\br^2+\om(r^2,y)\om(\br^1x, r^1)\br^2+\om(x,r^1)\om(\br^1y,r^2)\br^2\\
 &\stackrel{(\ref{eq:symp})}{=}&\om(x, r^1)\om(y, \bar{r}^1)r^2\bar{r}^2-\om(x, \br^1 r^2)\om(y,r^1)\br^2+\om(x, r^1)\om(y,r^2\br^1)\br^2-\om(x,r^1)\om(y,\br^1 r^2)\br^2\\
 &\stackrel{(\ref{eq:uj})}{=}&0. 
 \end{eqnarray*}
 These finish the proof.
 \end{proof}

 \begin{ex}\mlabel{ex:thm:ln} Let $(\gg,\cdot)$ be a 2-dimensional perm algebra with basis $\{e_1,e_2\}$ and the product $\cdot$ given by
\begin{center}
        \begin{tabular}{r|rr}
          $\cdot$ & $e_1$  & $e_2$   \\
          \hline
           $e_1$ & $e_1$  & $0$ \\
           $e_2$ & $e_2$  & $0$  \\
        \end{tabular}.
        \end{center}
       Set $r=\l e_1\o e_1$ ($\l$ is a parameter), then $(\gg,\cdot,r,\D_r)$ is a quasitriangular perm bialgebra, where $\D_r$ is given as below:
        $$\left\{
            \begin{array}{l}
             \D_r(e_1)=\l e_1\o e_1\\
             \D_r(e_2)=\l(e_1\o e_2+e_2\o e_1)\\
            \end{array}
            \right..$$
 Set
 \begin{center}
        \begin{tabular}{r|rr}
          $\omega$ & $e_1$  & $e_2$   \\
          \hline
           $e_1$ & $\nu$  & $0$  \\
           $e_2$ & $0$  & $0$  \\
        \end{tabular}\quad ($\nu$ is a parameter),
        \end{center}
 then $(\gg,\D_r,\omega,\cdot_\omega)$ is a dual quasitriangular perm bialgebra. Then by Theorem \ref{thm:ln}, let
   $$\left\{
            \begin{array}{l}
             N(e_1)=\l\nu e_1\\
             N(e_2)=0\\
            \end{array}
            \right., $$
 then $(\gg,  N)$ is a Nijenhuis perm algebra.
 \end{ex}

 \N operators on perm coalgebra can be derived by the following way.
 \begin{defi}\mlabel{de:csym} Let $(\gg, \D)$ be a perm coalgebra and $r\in \gg\o \gg$ be a symmetric element. Assume that
 \begin{eqnarray*}\mlabel{eq:csymp1}
 {r^1}_{(1)}\o {r^1}_{(2)}\o r^2+{r^1}_{(1)}\o {r^2}\o {r^1}_{(2)}-{r^1}_{(2)}\o r^2\o {r^1}_{(1)}-{r^2}\o {r^1}_{(2)}\o {r^1}_{(1)}=0.
 \end{eqnarray*}
 Then $(\gg, \D)$ is called {\bf a cosymplectic perm coalgebra} denoted by $(\gg, \D, r)$.
 \end{defi}

 \begin{pro}\mlabel{pro:de:csym}
 Let $(\gg, \cdot)$ be a perm algebra and $r\in \gg\o \gg$ symmetric. If $(\gg, \cdot, r, \D_r)$ is a quasitriangular perm bialgebra, where $\D_r$ is defined by Eq.(\ref{eq:cop}), then $(\gg, \D_r, r)$ is a cosymplectic perm coalgebra.
 \end{pro}

 \begin{proof} It is direct by Eqs.(\ref{eq:uj}) and (\ref{eq:cop}).
 \end{proof}

 \begin{thm}\mlabel{thm:cln} Let $(\gg, \D, r)$ be a cosymplectic perm coalgebra and $\om\in (\gg\o \gg)^*$. If $(\gg, \D, \om$, $\cdot_\om)$ is a dual quasitriangular perm bialgebra together with the multiplication $\cdot_\om$ defined in Eq.(\ref{eq:p}) and further $(\gg, \cdot_\om, r, \D_r)$ is a quasitriangular perm bialgebra together with the comultiplication $\D_r$ defined in Eq.(\ref{eq:cop}).
 Then $(\gg, \D, S)$ is a \N perm coalgebra, where $S:\gg\lr \gg$ is given by
 \begin{eqnarray*}
 &S(x)=r^1\om(r^2, x),~~\forall~x\in \gg.&\mlabel{eq:cys}
 \end{eqnarray*}
 \end{thm}

 \begin{proof} Similar to Theorem \ref{thm:ln}.
 \end{proof}

 \begin{ex}\mlabel{ex:thm:xc} Let $(\gg,\D)$ be a 2-dimensional perm coalgebra with basis $\{e_1,e_2\}$ and the coproduct $\D$ given by
 $$\left\{
            \begin{array}{c}
           \D(e_1)=e_1\o e_1\\
           \D(e_2)=e_2\o e_1\\
            \end{array}
            \right..$$
            
         Set
 \begin{center}
        \begin{tabular}{r|rr}
          $\omega$ & $e_1$  & $e_2$   \\
          \hline
           $e_1$ & $\nu$  & $0$  \\
           $e_2$ & $0$  & $0$  \\
        \end{tabular}\quad,
        \end{center}
 then $(\gg,\D,\omega,\cdot_\omega)$ is a dual triangular perm bialgebra with the multiplication $\cdot_\omega$ :
 \begin{center}
        \begin{tabular}{r|rr}
          $\cdot_\omega$ & $e_1$  & $e_2$   \\
          \hline
           $e_1$ & $\nu e_1$  & $\nu e_2$  \\
           $e_2$ & $\nu e_2$  & $0$  \\
        \end{tabular},
        \end{center}
 where $\nu$ is a parameter.

 Set an element $r\in \gg\o \gg$ by
   \begin{eqnarray*}
  r=\l e_1\o e_1,
  \end{eqnarray*}
  where $\l$ is parameter, then $(\gg,\cdot_\omega,r,\D_r)$ is a triangular perm bialgebra. Then by Theorem \ref{thm:cln}, let
   $$\left\{
            \begin{array}{l}
             S(e_1)=\l\nu e_1\\
             S(e_2)=0\\
            \end{array}
            \right., $$
 then $(\gg, \D, S)$ is a Nijenhuis perm coalgebra.
 \end{ex}
\vskip-6mm

 \section*{Acknowledgment} The authors are deeply indebted to the referee and the editor for their very useful suggestions and some improvements to the original manuscript. This work is supported by National Natural Science Foundation of China (No. 12471033) and Natural Science Foundation of Henan Province (No. 242300421389).
\bigskip

\noindent
 {\bf Author Declarations} The authors have no conflicts to disclose.
 
\noindent
 {\bf Data Availability Statement} No data was used for the research described in the article.

\end{document}